\documentclass[letterpaper,11pt,oneside]{article}
\usepackage{amsmath,amsthm}
\usepackage{amsfonts}
\usepackage{latexsym}
\usepackage{multicol}
\usepackage[latin1]{inputenc}
\usepackage{amssymb}
\usepackage{oldgerm}
\usepackage{enumerate}
\usepackage{fancyvrb}
\theoremstyle{plain}

\newtheorem{lem}{Lemma}

\newtheorem{thm}{Theorem}
\newtheorem{cor}{Corollary}
\newtheorem{rem}{Remark}
\newtheorem{example}{Example}
\title{$B_{h}[g]$ modular sets from $B_{h}$ modular sets}
\author{Nidia Y. Caicedo$^{*}$ \hspace{3cm}  Carlos A. Gómez$^{*}$ \\  Jhonny C. G\'omez$^{*}$ \hspace{3cm} Carlos A. Trujillo$^{**}$\\{\small $^{*}$Universidad del Valle, A.A. 25360, Colombia.}\\{\small $^{**}$Universidad del Cauca, A.A. 755, Colombia}.\\{\scriptsize E-mail: nidia.caicedo@correounivalle.edu.co, carlos.a.gomez@correounivalle.edu.co,}\\{\scriptsize jhonny.gomez@correounivalle.edu.co, trujillo@unicauca.edu.co}}
\date{}

\begin{document}
\maketitle

\begin{abstract}
\noindent
A set of positive integers $A$ is called a $B_{h}[g]$ set if there are at most $g$ different sums of $h$ elements from $A$ with the same result. This \linebreak definition has a generalization to abelian groups and the main \linebreak problem related to this kind of sets, is to find $B_{h}[g]$ maximal sets i.e. those with larger cardinality.
\\

We construct $B_{h}[g]$ modular sets from $B_{h}$ modular sets using \linebreak homomorphisms and analyze the constructions of $B_{h}$ sets by Bose and Chowla, Ruzsa, and Gómez and Trujillo look at for the suitable \linebreak homomorphism that allows us to preserve the cardinal of this types of sets.
\end{abstract}

\emph{Key words and phrases}. Sidon sets, $B_{h}$ sets, $B_{h}[g]$ sets, finite fields, homomorphism of groups.

\emph{2010 Mathematics Subject Classification}: 11B50, 12E20, 20K01, 20K30.
\section{Introduction}
Let $G$ be an abelian group with additive notation, $A\subseteq G$ numerable, and $h,g\in\mathbb{Z}^{+}$, with $h\geq2$. $A$ is a $B_{h}[g]$ set
on $G$, denoted by $A\in B_{h}[g](G)$, if each $b\in G$ has at most $g$ distinct representations as
\[
b=a_{i_{1}}+\cdots+a_{i_{h}},\mbox{ }\mbox{ }\mbox{ with }a_{i_{1}},\ldots,a_{i_{h}}\in A\mbox{ }\mbox{ }\mbox{ y }\mbox{ }i_{1}\leq\ldots\leq i_{h}.
\]
When $g=1$ we say that $A\in B_{h}(G)$ and, in particular, when
$h=2$ and $G=\mathbb{Z}$ is used to talk about $B_{2}$ sets of
integers or Sidon Sets so named in honor of Simon Sidon, an Hungarian
analyst who needed to find subsets of positive integers numbers with
the property that all sums of two elements get a different result,
except commutativity. Some others researches who worked with Sidon
sets were P. Erd\"os, P. Turán, J. Singer, R. C. Bose, S. Chowla, I.
Z. Ruzsa, B. Lindstr\"om, J. Cilleruelo, C. Trujillo, B. Green, C. Vinuesa and others.

The fundamental problem related to this sets is to find
the bigger cardinal of a $B_{h}[g]$ contained in a group $G$. Pointing
in this direction, we define the function
\[
f_{h}(G,g):=\max\{|A|:A\in B_{h}[g](G)\},
\]
which has not been determined entirely yet for each $h,g\in\mathbb{Z}^{+}$
and $G$ an additive abelian group. There are two different ways to approach the goal of this problem, using counting methods in order to obtain upper bounds and the other that provides lower bounds for  $f_{h}(G,g)$ by the construction of ``good'' examples of $B_{h}[g]$ sets i.e. sets with a large cardinality.

In this paper we present a new way to construct a $B_{h}[g]$ set from a $B_{h}$
set known, that allows us to provide some lower bounds for $f_{h}(G,g)$ by applying our main result to some classic constructions. When $G=\mathbb{Z}_{N}$ we talk of modular $B_{h}[g]$ sets and customary notation $f_{h}(N,g)$.

Martin and O'Bryant \cite{KO} shows the next lower bounds for $f_2(N,g)$ using a prime
number $p$, a prime power $q$ and a positive integer $k$:
\begin{enumerate}
\item[(a)] $f_{2}(p^{2}-p,k^{2})\geq k(p-1)$,
\item[(b)] $f_{2}(q^{2}-1,k^{2})\geq kq$,
\item[(c)] $f_{2}(q^{2}+q+1,k^{2})\geq kq+1$.
\end{enumerate}
In order to establish the difference with our main result is to highlighting that this bounds are valid only for square perfect $g$ and for $h=2$,
i.e.  $B_{2}[g]$ sets in a modular group
with $g=k^{2}$ for $k\in\mathbb{Z}^{+}$. Martin and O'Bryant established
this bounds when they constructed generalized Sidon sets by joining
suitable $B_{2}$ sets type Singer \cite{SINg},
type Bose \cite{BO} and type Ruzsa \cite{RZ}.

Our approach is different than Martin and
O'Bryant because we use a particular result from Lemma \ref{LemaPrincipal},
which allows us to conclude that for any positive integer $g$:
\begin{enumerate}
\item[(a)] For any prime number $p$, such that $p\equiv1(\bmod\, g)$:
\begin{equation}\label{cota inf ruzsa}
f_{2}\left(\frac{p^{2}-p}{g},g\right)\geq p-1.
\end{equation}

\item[(b)] For any prime number $p$ and any integer $h\geq3$, such that \linebreak $p^{h}\equiv1(\bmod\, g)$:
%for some positive integer $d$ that divides $h-1$ but $d\not=h-1$:
\begin{equation}\label{cota inf bose y chowla}
f_{h}\left(\frac{p^{h}-p}{g},g\right)\geq p.
\end{equation}

\item[(c)] For any prime power $q$ and any integer $h\geq2$, such that \linebreak $q^{d}\equiv1(\bmod\, g)$
for some positive integer $d$ that divides $h$ but $d\not=h$:
\begin{equation}\label{cota inf G y T}
f_{h}\left(\frac{q^{h}-1}{g},g\right)\geq q.
\end{equation}

\end{enumerate}

Next, we show some cases where our results gives bigger lower bounds
than those obtained by Martin and O'Bryant for $f_{2}(N,g)$, although
the principal difference is that we give lower bounds for $f_{h}(N,g)$
with $h,g\in\mathbb{Z}^{+}$ and $h\geq2$:
\begin{itemize}
\item For $p=5$ and $k=2$, as $p^{2}-p=20$, according to Martin and O'Bryant
$f_{2}(20,4)\geq8$. But if we take $q=9$ and $g=4$, as $\frac{q^{2}-1}{g}=20$
then we get $f_{2}(20,4)\geq9.$
\item For $p=7$ and $k=2$, as $p^{2}-p=42$ then $f_{2}(42,4)\geq12$.
But if $q=13$ and $g=4$, as $\frac{q^{2}-1}{g}=42$, we get $f_{2}(42,4)\geq13$.
\end{itemize}
When $G=\mathbb{Z}$, it is common to denote the function
\[
F_{h}(N,g):=\max\{|A|:\mbox{ }A\subseteq[1,N],\mbox{ }A\in B_{h}[g](\mathbb{Z})\}.
\]
A lower bound of this function was prove by Lindstr\"om \cite{LN},
who established that
$$
F_{h}(N,g)\geq(1+o(1))\sqrt[h]{gN},
$$
 with $g=m^{h-1}$ for an integer $m\geq2$, when $N\rightarrow\infty$.

Now, since every $B_{h}[g]$ modular set is a $B_{h}[g]$ integer set we get
$$F_{h}(N,g)\geq f_{h}(N,g),$$
and therefore the previous bounds on $f_h(N,g)$ also holds for
$F_{h}(N,g)$. So, considering $N=\frac{p^h-1}{g}$, we have from (\ref{cota inf G y T}) that
$$F_{h}(N,g)\geq\sqrt[h]{gN+1}$$
for an infinite number of values of $N$, according to  Dirichlet Theorem about prime numbers in arithmetic progressions.
%This shows that using our main result we could obtain a better lower bound than Lindstr\"om bound.

\section{The main result and the $B_{h}$ sets constructions}

The next Lemma is our main result and this allow us to establish what
happen when we apply a group homomorphism to a $B_{h}$ set. Using
this we obtained a way to construct a $B_{h}[g]$ set from a $B_{h}$
set already known.
\begin{lem}
\label{LemaPrincipal}
Let $G$ and $G^{\prime}$ be finite  additive abelian groups and $\phi:G\rightarrow G^{\prime}$
be \linebreak  a homomorphism such that $|Ker(\phi)|=g_{1}$. If $A\in B_{h}[g](G)$, then \linebreak $\phi(A)\in B_{h}[gg_{1}](\phi(G))$.
\end{lem}

\begin{proof}
Let us proceed by reduction to absurd.  We suppose that an element $\alpha\in\phi(G)$
has $gg_{1}+1$ different representations like an $h-$sum of elements
from $\phi(A)$; by means,
\[
\alpha=\sum_{i=1}^{h}\phi(a_{i,k}),\mbox{ }\mbox{ }\mbox{ }\mbox{ for }\mbox{ }k=1,\ldots,gg_{1}+1,
\]
where each $a_{i,k}\in A$, for $i=1,\ldots,h$, for $k\not=k^{\prime}$. Also, that
$$\{\{a_{i,k}:1\leq i\leq h\}\}\not=\{\{a_{i,k^{\prime}}:1\leq i\leq h\}\}\footnote{The notation $\{\{\}\}$ means multiset.}.$$

Now, if we denote
\[
\alpha_{k}:=\sum_{i=1}^{h}\phi(a_{i,k}),
\]
then $\alpha_{k}=\alpha$ for every $k=1,\ldots,gg_{1}+1$. Let us
fix $\alpha_{1}$, thus
\[
\alpha_{1}-\alpha_{j}=0\mbox{ }\mbox{ }\mbox{for }\mbox{ }j=2,\ldots,gg_{1}+1,
\]
but, on the other hand,
\[
\begin{array}{rcl}
\alpha_{1}-\alpha_{j} & = & \sum_{i=1}^{h}\phi(a_{i,1})-\sum_{i=1}^{h}\phi(a_{i,j})\\ \\
 & = & \phi\left(\sum_{i=1}^{h}(a_{i,1}-a_{i,k})\right).
\end{array}
\]
We put $b_{k}$  to denote
\[
\sum_{i=1}^{h}(a_{i,1}-a_{i,k}), \text{ for } k=2,\ldots,gg_{1}+1,
\]
we have $\phi(b_{k})=0$ or, what is equivalent, $b_{k}\in Ker(\phi)$
for every \linebreak $k=2,\ldots,gg_{1}+1$, and since $|Ker(\phi)|=g_{1}$,
by the pigeonhole principle, there are $b_{k_{j}}=0$ for $j=1,\ldots,g$
or $b_{k_{j}}=t\not=0$, for $j=1,\ldots,g+1$; each case implies
there is an element in $G$ that has $g+1$ different representations
as $h-$sums of elements in $A$, contradicting that $A\in B_{h}[g](G)$.
\end{proof}

\begin{cor}
\label{CorTheoBhg}
Let $G$ and $G^{\prime}$ finite additive abelian
groups and \linebreak $\phi:G\rightarrow G^{\prime}$ be a injective homomorphism.
If $A\in B_{h}[g](G)$, then \linebreak $\phi(A)\in B_{h}[g](G^{\prime})$.
\end{cor}

In order to obtain lower bounds for $f_{h}(G,g)$ we  establish conditions that allows us to
to preserve the cardinal of a set when we apply a homomorphism, more precisely we have the next lemma.
%it is why we present the next lemma whose proof is not hard to do.

\begin{lem}
\label{Lem|Bhg|}
Let $\phi:G\rightarrow G^{\prime}$ be a finite additive abelian groups \linebreak homomorphism. Then
\[
|\phi(A)|=|A|\Longleftrightarrow(A-A)\cap Ker(\phi)=\{0\}.
\]
\end{lem}
The next construction, due to Derksen  \cite{DE}, use a quotient ring of \linebreak polynomials
to construct $B_{h}$ sets, that allows us via Corollary \ref{CorTheoBhg} a \linebreak different presentation
of some  classic constructions, of modular $B_h$ sets.

From now on, given any ring $R$, we denote the group of units of $R$ by $R^{*}$.

\begin{lem}
\label{LemGenBh}
Let $F$ be a field and $p(x)\in F[x]$ with $deg(p(x))=h$. If $p(s)\not=0$ for every $s\in S$, then
\[
x-S=\{x-s:\mbox{ }s\in S\},
\]
is a $B_{h}$ set at $\left({F[x]}\big\slash{\langle p(x)\rangle}\right)^{\ast}$
with $|S|$ elements.
\end{lem}

\begin{proof}
For $s\in S$, we have $(x-s)\in\left({F[x]}\big\slash{\langle p(x)\rangle}\right)^{\ast}$
if and only if \linebreak $gcd(x-s,p(x))=1$, but this is equivalent to have
$p(s)\not=0$. \\

Assume that there is an element belonging $\left({F[x]}\big\slash{\langle p(x)\rangle}\right)^{\ast}$
which has two different representations as the product to $h$ elements
from $x-S$; in other words,
\[
(x-s_{1})\cdots(x-s_{h})\equiv(x-s_{1}^{\prime})\cdots(x-s_{h}^{\prime})(\bmod\, p(x)),
\]
where $s_{i},s_{i}^{\prime}\in S$ for all $i=1,\ldots,h$. If we
use $f(x)$ to denote \linebreak $\prod_{i=1}^{h}(x-s_{i})-\prod_{i=1}^{h}(x-s_{i}^{\prime})$,
then $p(x)|f(x)$. But $deg(f(x))=h-1$, so
\[
\prod_{i=1}^{h}(x-s_{i})=\prod_{i=1}^{h}(x-s_{i}^{\prime}).
\]
However, $F[x]$ is a unique factorization domain, so $$\{s_{i}:1\leq i\leq h\}=\{s_{i}^{\prime}:1\leq i\leq2\}.$$
\end{proof}

\begin{rem}
As a particular case, if $\theta$ is an algebraic element of degree
$h$ over the field $F$, then the set
\[
\theta+F:=\{\theta+a:a\in F\}
\]
is a $B_{h}$ set in the multiplicative group \emph{$(F(\theta))^{\ast}$.
}Even more, when \linebreak $F=\mathbb{F}_{q}$, we can conclude that $BC(q,\theta):=\theta+\mathbb{F}_{q}\in B_{h}(\mathbb{F}_{q^{h}}^{\ast})$
with $q$ \linebreak elements. More generally, if we consider an algebraic element
$\beta$ with \linebreak degree $d$ over $\mathbb{F}_{q}$, then $BC(q,\beta)\in B_{d}(\mathbb{F}_{q^{h}}^{\ast})$ provided $d|h$.
\end{rem}

\begin{cor}
\label{CorBoChow}
Let $q$ be a prime power, $h\geq2$
an integer and $\theta$ a primitive element of $\mathbb{F}_{q^{h}}$.
The set
\[
B(q,h,\theta):=\log_{\theta}(\theta+\mathbb{F}_{q})(\bmod\, q^{h}-1)
\]
 is a $B_{h}$ set in $\mathbb{Z}_{q^{h}-1}$ with $q$ elements.
\end{cor}
This construction can be seen as a consequence from the previous results.
Just by taking into consideration the isomorphism  $\log_{\theta}:\mathbb{F}_{q^{h}}^{\ast}\rightarrow\mathbb{Z}_{q^{h}-1}$
(called the \emph{discrete logarithm} to the base $\theta$), defined by $\log_{\theta}(x)=a$
if and only if $\theta^{a}=x$, and $\log_{\theta}(A):=\{\log_{\theta}(a):a\in A\}$.
Now, using Corollary \ref{CorTheoBhg} we have that $B(q,h,\theta)\in B_{h}\left(\mathbb{Z}_{q^{h}-1}\right)$.

The classic construction by Bose and Chowla  \cite{B&C}, gives
the set 
\[
B\&C(q,h,\theta):=\{a\in[1,q^{h}-1]:\theta^{a}-\theta\in\mathbb{F}_{q}\}
\]

that is equal to $B(q,h,\theta)$. In fact, $a\in B\&C(q,h,\theta)$
if and only if \linebreak $\theta^{a}-\theta\in\mathbb{F}_{q}$, which is equivalent
to have $\theta^{a}=\theta+x$ for some $x\in\mathbb{F}_{q}$, if
and only if $a=\log_{\theta}(\theta+x)$ which means that $a\in B(q,h,\theta)$.

Now, we want to show the way to build a  $B_{d}$ multiplicative set  in \linebreak
$\mathcal{H}={\mathbb{F}_{q^{h+1}}^{\ast}}\big\slash{\mathbb{F}_{q}^{\ast}}$
with $q+1$ elements, known as the generalization of Singer's construction
and for this, in the next theorem we consider for $d\geq2$ the set
\[
BC(q,\beta)\in B_{d+1}\left(\mathbb{F}_{q^{h+1}}^{\ast}\right),
\]
with $q$ elements, where $\beta$ is an algebraic element in
$\mathbb{F}_{q^{h+1}}$ of degree $d+1$ over $\mathbb{F}_{q}$. Also, we consider
the homomorphism
\[
\begin{array}{cccc}
\phi: & \mathbb{F}_{q^{h+1}}^{\ast} & \rightarrow & \mathcal{H}\\
 & x & \mapsto & \overline{x}:=x+\mathbb{F}_{q}^{\ast}
\end{array}.
\]

\begin{thm}
For each element $\beta\in\mathbb{F}_{q^{h+1}}$, of degree $d+1$ over $\mathbb{F}_{q}$
the set
\[
SG(q,\beta):=\{\overline{1}\}\cup\phi(BC(q,\beta))%\in B_{d}\left(\mathcal{H}\right)
\]
is a $B_{d}\left(\mathcal{H}\right)$ set with $q+1$ elements.
\end{thm}
\begin{proof}
Let $\overline{\beta+i}=\overline{\beta+j}$, as $i,j\in\mathbb{F}_{q}$
this implies $i=j$ therefore \linebreak $|\phi(BC(q,\beta))|=q$. Besides, as
$\beta$ has degree $d+1>1$, then
\[
\overline{1}\not\in\phi(BC(q,\beta)),
\]
so $|SG(q,\beta)|=q+1$.

Now, we are going to show that $\phi(BC(q,\beta))\in B_{d}\left(\mathcal{H}\right)$.
Let us assume there are $\{\{a_{i}:1\leq i\leq d\}\}$ and $\{\{b_{j}:1\leq j\leq d\}\}$
such that
\begin{equation}\label{prod 1}
\prod_{i=1}^{d}\overline{\left(\beta+a_{i}\right)}\equiv\prod_{j=1}^{d}\overline{\left(\beta+b_{j}\right)}\mbox{ }\mbox{ }\mbox{ }\mbox{ in }\mbox{ }\mathcal{H},
\end{equation}

then, there is $t\in\mathbb{F}_{q}^{\ast}$ such that
\[
\prod_{i=1}^{d}\left(\beta+a_{i}\right)=t\prod_{i=1}^{d}\left(\beta+b_{i}\right)\mbox{ }\mbox{ }\mbox{ }\mbox{ in }\mbox{ }\mathbb{F}_{q^{h+1}}^{\ast},
\]
since the degree of $\beta$ is $d+1$ over $\mathbb{F}_{q}$, (\ref{prod 1})
is possible provided $t=1$, and $\{\{a_{i}:1\leq i\leq d\}\}=\{\{b_{j}:1\leq j\leq d\}\}$. So
\[
\{\{\overline{\beta+a_{i}}:1\leq i\leq d\}\}=\{\{\overline{\beta+b_{j}}:1\leq j\leq d\}\}.
\]
%We need to verify if we add $\overline{1}$ to $\phi(BC(q,\beta))$ the set still is a $B_{d}$ set.
It remains to prove the case in which some terms in (\ref{prod 1}) are $\overline{1}$. In order to do this, let us assume
there is a element of $\mathcal{H}$ with two different representations
as the product of $d$ elements of $SG(q,\beta)$. Without loss of
generality, let us assume that in one of this representations there
is $k$ times the $\overline{1}$; i.e.,
\[
\prod_{i=1}^{d}\overline{\left(\beta+a_{i}\right)}\equiv\prod_{j=1}^{d-k}\overline{\left(\beta+b_{j}\right)}\mbox{ }\mbox{ }\mbox{ }\mbox{ in }\mbox{ }\mathcal{H},
\]
where $\{\{a_{i}:1\leq i\leq d\}\}$ and $\{\{b_{j}:1\leq j\leq d\}\}$ are
contained in $\mathbb{F}_{q}$. But no matter the value of $k$, this
leaves to a contradiction with the degree of $\beta$, therefore $SG(q,\beta)\in B_{d}\left(\mathcal{H}\right)$.
\end{proof}
\begin{rem}
Note that $\mathcal{H}\cong\mathbb{Z}_{N}$ with $N=\frac{q^{h+1}-1}{q-1}$. Then by Corollary \ref{CorTheoBhg}, if $\beta$ is an element of degree $h+1$ over $\mathbb{F}_{q}$, we obtained the generalized Singer's construction.
\end{rem}

Finally, we are going to use the classic constructions of $B_h$ sets type Ruzsa, Bose-Chowla, Gómez-Trujillo to built $B_h[g]$ sets, via Corollary \ref{CorTheoBhg}.

\section{Some consequences of our main result.}

The next theorem is a consequence of the main result in this paper,
where we consider a $B_{2}$ set constructed by Ruzsa called a $B_{2}$
type Ruzsa set.

\begin{thm}
\label{TheoAplReultPrinRuzsa}
For all $g\in\mathbb{Z}^{+}$, there is a $B_{2}[g]$ set in $\mathbb{Z}_{\frac{p^{2}-p}{g}}$
with $p-1$ elements, where $p$ is a prime number such that $p\equiv1(\bmod\, g)$.
\end{thm}

\begin{proof}
Let us take $\theta$ a primitive element of $\mathbb{Z}_{p}$, and the $B_2$ set
\[
R(p,\theta):=\{(a,\mbox{ }\theta^{a}):a\in\{1,2\ldots,p-1\}\}\subseteq \mathbb{Z}_{p-1}\times\mathbb{Z}_{p}
\]
with $p-1$ elements. Let us consider $H_{g}:=\langle\frac{p-1}{g}\rangle$ as
a subgroup of $\mathbb{Z}_{p-1}\cong\mathbb{Z}_{p}^{\ast}$ and the homomorphism
\[
\begin{array}{cccc}
\phi: & \mathbb{Z}_{p-1}\times\mathbb{Z}_{p} & \rightarrow & \left({\mathbb{Z}_{p-1}}\big\slash{H_{g}}\right)\times\mathbb{Z}_{p}\\
 & (x,y) & \mapsto & (x+H_{g},y)
\end{array}
\]
then $|Ker(\phi)|=g$ and, by Lemma \ref{LemaPrincipal},
\[
\phi(R(p,\theta))\in B_{2}[g]\left(\left({\mathbb{Z}_{p-1}}\big\slash{H_{g}}\right)\times\mathbb{Z}_{p}\right).
\]
But the canonical isomorphism establish that
\[
\left({\mathbb{Z}_{p-1}}\big\slash{H_{g}}\right)\times\mathbb{Z}_{p}\cong\mathbb{Z}_{\frac{p-1}{g}}\times\mathbb{Z}_{p},
\]
and, as $gcd\left(\frac{p-1}{g},p\right)=1$, by the Chinese Remainder
Theorem, \linebreak $\mathbb{Z}_{\frac{p-1}{g}}\times\mathbb{Z}_{p}\cong\mathbb{Z}_{\frac{p^2-p}{g}}$.

Using Lemma \ref{Lem|Bhg|}, we are going to show that $|\phi(R(p,\theta))|=p-1$.
Let us denote
\[
\alpha_{i}:=(a_{i},\theta^{a_{i}})\in R(p,\theta),\mbox{ }\mbox{ }\mbox{ }i=1,2,
\]
with $\alpha_{2}-\alpha_{1}\in Ker(\phi)$, then $\phi(\alpha_{2})=\phi(\alpha_{1})$,
i.e.,
\[
(a_{1}+H_{g},\mbox{ }\theta^{a_{1}})=(a_{2}+H_{g},\mbox{ }\theta^{a_{2}}).
\]
So, $a_1-a_2\in H_{g}$ implying that $\frac{p-1}{g}\big|(a_1-a_2)$; besides, $\theta^{a_1}\equiv\theta^{a_2}(\bmod\, p)$, then $p|a_1-a_2$ and $a_1-a_2\leq p-1$,
which implies $a_{1}=a_{2}$ and therefore $$(R(p,\theta)-R(p,\theta))\cap Ker(\phi)=\{0\}.$$

$Ruzsa(p,\theta)\in B_{2}\left(\mathbb{Z}_{p^{2}-p}\right)$ is the
image of $R(p,\theta)$ via the Chinese Remainder Theorem, which is
the classic construction of Ruzsa, \cite{RZ}. \\
\end{proof}
Let $\varphi$ be an homomorphism between $\mathbb{Z}_{p^{2}-p}$
and $\mathbb{Z}_{\frac{p^{2}-p}{g}}$ such that \linebreak $|ker(\varphi)|=g$,
then by Lemma \ref{LemaPrincipal}
\[
Ruzsa(p,\theta,g)=\varphi(Ruzsa(p,\theta))\in B_{2}[g]\left(\mathbb{Z}_{\frac{p^{2}-p}{g}}\right).
\]

\begin{cor}
\label{CoroResultPrinpRuzsa} \emph{For all $g\in\mathbb{Z}^{+}$,
if $p$ is a prime number such that \linebreak $p\equiv1(\bmod\, g)$, then
\[
f_{2}\left(\frac{p^{2}-p}{g},g\right)\geq p-1.
\]
}
\end{cor}
Now, we use the construction of a $B_{h}$ set type Gómez and Trujillo
set to apply our main result.

\begin{thm}
\emph{\label{TheoAplResulPrinGT} For all $g\in\mathbb{Z}^{+}$ there
is a $B_{h}[g]$ set at $\mathbb{Z}_{\frac{p^{h}-p}{g}}$ with $p$
elements, where $p$ is a prime number and $h\geq3$ is an integer such
that $p^{h}\equiv1(\bmod\, g)$.}
%for some positive integer $d$ that divides $h-1$ but \linebreak $d\not=h-1$.}
\end{thm}

\begin{proof}
At \cite{GT-1} shows that, for $\theta$ a primitive element of $\mathbb{F}_{p^{h-1}}$,
\[
GT(p,h,\theta)=\{(a,\log_{\theta}(\theta+a):a\in\mathbb{Z}_{p}\}
\]
is a $B_{h}(\mathbb{Z}_{p}\times\mathbb{Z}_{p^{h-1}-1})$ set
with $p$ elements. Consider $H_{g}:=\langle\frac{p^{h-1}-1}{g}\rangle$
subgroup of $\mathbb{Z}_{p^{h-1}-1}$ and the homomorphism
\[
\begin{array}{cccc}
\phi: & \mathbb{Z}_{p}\times\mathbb{Z}_{p^{h-1}-1} & \rightarrow & \mathbb{Z}_{p}\times\left({\mathbb{Z}_{p^{h-1}-1}}\big\slash{H_{g}}\right)\\
 & (x,y) & \mapsto & (x,y+H_{g})
\end{array}
\]
then $|Ker(\phi)|=g$ and by Lemma \ref{LemaPrincipal}, we have $\phi(GT(p,h,\theta))\in B_{h}[g]$
at $\mathbb{Z}_{p}\times\left({\mathbb{Z}_{p^{h-1}-1}}\big\slash{H_{g}}\right)$,
but, using the canonical isomorphism
\[
\mathbb{Z}_{p}\times\left({\mathbb{Z}_{p^{h-1}-1}}\big\slash{H_{g}}\right)\cong\mathbb{Z}_{p}\times\mathbb{Z}_{\frac{p^{h-1}-1}{g}},
\]
and, the Chinese Remainder Theorem states that $\mathbb{Z}_{p}\times\mathbb{Z}_{\frac{p^{h-1}-1}{g}}\cong\mathbb{Z}_{\frac{p^{h}-p}{g}}$.

Now, using Lemma \ref{Lem|Bhg|}, let
\[
\alpha_{i}:=(a_{i},\log_{\theta}(\theta+a_{i}))\in GT(p,h,\theta),\mbox{ }\mbox{ }i=1,2,
\]
if $\alpha_{2}-\alpha_{1}\in Ker(\phi)$, then $\phi(\alpha_{2})=\phi(\alpha_{1})$,
i.e.,
\[
(a_{1},\mbox{ }\log_{\theta}(\theta+a_{1})+H_{g})=(a_{2},\mbox{ }\log_{\theta}(\theta+a_{2})+H_{g}),
\]
but this implies $a_{1}=a_{2}$, therefore $|\phi(GT(p,h,\theta))|=p$\emph{. }
\end{proof}
Now, $G\&T(p,h,\theta)\in B_{h}\left(\mathbb{Z}_{p^{h}-p}\right)$
is the image of $GT(p,h,\theta)$ via the \linebreak Chinese Remainder Theorem.
If $\varphi$ is the appropriate homomorphism between $\mathbb{Z}_{p^{h}-p}$
and $\mathbb{Z}_{\frac{p^{h}-p}{g}}$ with $|ker(\varphi)|=g$ then
by Lemma \ref{LemaPrincipal} we conclude that
\[
G\&T(p,h,\theta,g)=\varphi(G\&T(p,h,\theta))\in B_{h}[g]\left(\mathbb{Z}_{\frac{p^{h}-p}{g}}\right).
\]

\begin{cor}
\emph{\label{CoroResulPrinGT}For all $g\in\mathbb{Z}^{+}$, if $p$
is a prime number and $h\geq3$ is an integer such that $p^{h}\equiv1(\bmod\, g)$,
then
\[
f_{h}\left(\frac{p^{h}-p}{g},g\right)\geq p.
\]
}
\end{cor}
At last but not least important, we use a $B_{h}$ set type Bose and Chowla
to obtain a $B_{h}[g]$ set in \emph{$\mathbb{Z}_{\frac{q^{h}-1}{g}}$. }

\begin{thm}
\emph{\label{TheoAplResultPrinB&C} For all $g\in\mathbb{Z}^{+}$,
there is a $B_{h}[g]$ set in $\mathbb{Z}_{\frac{q^{h}-1}{g}}$
with $q$ elements, where $q$ is a prime power and $h\geq2$ is an integer,
such that $q^{k}\equiv1(\bmod\, g)$ for some $k$ that divides $h$,
but $k\not=h$. }\end{thm}

\begin{proof}
Consider the set $\theta+\mathbb{F}_{q}$, where $\theta$ is such
that $\left\langle \theta\right\rangle =\mathbb{F}_{q^{h}}^{\ast}$.
If we denote by $H_{g}$ the subgroup of $\mathbb{F}_{q^{k}}^{\ast}=\langle\theta^{\frac{q^{h}-1}{q^{k}-1}}\rangle$
generated by $\theta^{\frac{q^{h}-1}{g}}$, then the natural homomorphism
\[
\begin{array}{cccc}
\phi: & \mathbb{F}_{q^{h}}^{\ast} & \rightarrow & {\mathbb{F}_{q^{h}}^{\ast}}\big\slash{H_{g}}\\
 & x & \mapsto & x+H_{g}
\end{array}
\]
have $Ker(\phi)=H_{g}$ and $|Ker(\phi)|=g$. From Lemma \ref{LemaPrincipal}, we have that \linebreak
$\phi(\theta+\mathbb{F}_{q})\in B_{h}[g]\left({\mathbb{F}_{q^{h}}^{\ast}}\big\slash{H_{g}}\right)$.

Let $a,b\in\mathbb{F}_{q}$, if $\phi(\theta+a)=\phi(\theta+b)$,
then
\[
\theta+a\equiv\theta+b(\bmod\, H_{g}),
\]
this implies $\theta+a\equiv\theta+b\,(\bmod\,\mathbb{F}_{q^{k}}^{\ast}),$
therefore $\theta+a=\alpha(\theta+b)$ for $\alpha\in\mathbb{F}_{q^{k}}^{\ast}$.
Now, if $\alpha=1$ then $a=b$, on contrary
\[
\theta=\frac{b\alpha-a}{\alpha-1}\in\mathbb{F}_{q^{k}}^{\ast},
\]
that contradicts $grad_{\mathbb{F}_{q}}(\theta)=h>k$, therefore $|\phi(\theta+\mathbb{F}_{q})|=q$

Finally, using the isomorphism known as the discrete logarithm to
the base $\theta$,
\[
{\mathbb{F}_{q^{h}}^{\ast}}\big\slash{H_{g}}\cong\mathbb{Z}_{\frac{q^{h}-1}{g}},
\]
and Corollary \ref{CorTheoBhg} guarantees the existence of a $B_{h}[g]$ set in $\mathbb{Z}_{\frac{q^{h}-1}{g}}$ with $q$ elements.\end{proof}

\begin{cor}
\emph{\label{CoroResulPrinB&C}For all $g\in\mathbb{Z}^{+}$, if $q$
is a prime power and $h\geq2$ is an integer, such that $q^{d}\equiv1(\bmod\, g)$,
for some $d$ that divides $h$ but $d\not=h$,
\[
f_{h}\left(\frac{q^{h}-1}{g},g\right)\geq q.
\]
}
\end{cor}

\section{ Some numeric examples}

Using the classic constructions, we present some examples by using our main result highlighting the advantages
and disadvantages of using appropriate homomorphisms to built $B_{h}[g]$
sets.
\begin{example}
Let $g=2$,  for $q=7$ and $h=3$. Since $7\equiv1(\bmod\,2)$, we have that
\[
B\&C(7,3,\theta,2)=\{1,\mbox{ }68,\mbox{ }96,\mbox{ }108,\mbox{ }123,\mbox{ }128,\mbox{ }149\}\in B_{3}[2]\left(\mathbb{Z}_{171}\right),
\]
because of the main result applied to the set
\[
B\&C(7,3,\theta)=\{1,\mbox{ }108,\mbox{ }123,\mbox{ }128,\mbox{ }149,\mbox{ }239,\mbox{ }267\}\in B_{3}\left(\mathbb{Z}_{342}\right),
\]
 where $\theta=-2x^2+2x-1$ is a primitive element of $\mathbb{F}_{7^3}$. Check
that
$$108+108+108\equiv68+128+128(\bmod\,171)$$ and $$1+1+1\equiv68+128+149(\bmod\,171).
$$
Now, for $g=3$ as $7\equiv1(\bmod\,3)$, using the previous $B\&C(7,3,\theta)$,
\[
B\&C(7,3,\theta,3)=\{1,\mbox{ }9,\mbox{ }11,\mbox{ }14,\mbox{ }35,\mbox{ }39,\mbox{ }108\}\in B_{3}[3]\left(\mathbb{Z}_{114}\right),
\]
but there is not elements from $\mathbb{Z}_{114}$ with three different
representations as a $3-$sum of elements from $B\&C(7,3,\theta,3)$,
at most there are some cases like
$$39+108+108\equiv9+9+9(\bmod\,114)$$ and $$11+14+39\equiv35+35+108(\bmod\,114),
$$
that is why, in fact $B\&C(7,3,\theta,3)\in B_{3}[2]\left(\mathbb{Z}_{114}\right).$\\

Similarly, for $g=6$ as $7\equiv1(\bmod\,6)$ and $B\&C(7,3,\theta)\in B_{3}\left(\mathbb{Z}_{342}\right)$
we have
\[
B\&C(7,3,\theta,6)=\{1,\mbox{ }9,\mbox{ }11,\mbox{ }14,\mbox{ }35,\mbox{ }39,\mbox{ }51\}\in B_{3}[6]\left(\mathbb{Z}_{57}\right),
\]
also in this case, there is not elements of $\mathbb{Z}_{57}$ with
six different representations as $3-$sum of elements from $B\&C(7,3,\theta,6)$,
but for direct verification is easy to verify that it is a $B_{3}[3]\left(\mathbb{Z}_{57}\right)$
set,
\[
1+1+1\equiv11+14+35\equiv39+39+39(\bmod\,57),
\]
\[
9+9+9\equiv+39+51+51\equiv14+35+35(\bmod\,57).
\]

\end{example}
Sinca a $B_{h}[g]$ is also a $B_{h}[g^{\prime}]$ for any $g^{\prime}>g$,
the previous example using the Bose and Chowla construction shows
that our main result, in some cases, allow us to construct a $B_{h}[g^{\prime}]$
for $g^{\prime}<g$ in $\mathbb{Z}_{\frac{q^{h}-1}{g}}$.
\begin{example}
Now, using $g=3$, $q=7$ and $h=3$, the set
\[
B\&C(7,3,\alpha,3)=\{1,\mbox{ }7,\mbox{ }24,\mbox{ }36,\mbox{ }38,\mbox{ }49,\mbox{ }54\}\in B_{3}[3]\left(\mathbb{Z}_{114}\right)
\]
obtained by apply the main result to the set
$$
B\&C(7,3,\alpha)=\{1,\mbox{ }121,\mbox{ }152,\mbox{ }168,\mbox{ }252,\mbox{ }264,\mbox{ }277\}\in B_{3}\left(\mathbb{Z}_{342}\right),
$$
where $\alpha=x^2-3x+3$ is a primitive element of $\mathbb{F}_{7^3}$ different
from the one used in the previous example. Is easy to check that
\[
1+1+36\equiv54+49+49\equiv7+7+24(\bmod\,114).
\]

\end{example}
This previous example shows a case where the set is in fact a $B_{h}[g]$
with $g=|Ker(\varphi)|$, where $\varphi$ is the appropriate homomorphism
to the $B\&C(7,3,\theta)$ set.
\begin{example}
Let $g=3$, since $7^{k}\equiv1(\bmod\,3)$ for $k=2$, we use $q=7$ and
$h=4$, then
\[
B\&C(7,4,\beta,3)=\{1,\mbox{ }275,\mbox{ }429,\mbox{ }449,\mbox{ }621,\mbox{ }644,\mbox{ }756\}\in B_{4}[3]\left(\mathbb{Z}_{800}\right),
\]
applying our main result to the set $$B\&C(7,4,\beta)=\{1,\mbox{ }429,\mbox{ }621,\mbox{ }644,\mbox{ }1249,\mbox{ }1556,\mbox{ }1875\}\in B_{4}\left(\mathbb{Z}_{2400}\right)$$
where $\beta=-2x^2+x+2$ is a primitive element of $\mathbb{F}_{7^4}$. But
in this case, the set is in fact a $B_{4}[2]$.

Let $g=24$, since $7^{k}\equiv1(\bmod\,24)$ for $k=2$, use $q=7$ and
$h=4$, then
\[
B\&C(7,4,\beta,24)=\{1,\mbox{ }21,\mbox{ }29,\mbox{ }49,\mbox{ }56,\mbox{ }44,\mbox{ }75\}\in B_{4}[24]\left(\mathbb{Z}_{100}\right),
\]
using the previous set $B\&C(7,4,\beta)$. The expected is a
$B_{4}[g^{\prime}]$ set with \linebreak $g^{\prime}<24$, and, as direct
verification shows it is in fact a $B_{4}[7]$.
\begin{align*}
1+21+29+49&\equiv44+44+56+56\equiv1+1+49+49 \\
&\equiv21+21+29+29\equiv75+75+75+75 \\
&\equiv1+49+75+75\equiv21+29+75+75(\bmod\,100)
\end{align*}

\end{example}
On the next examples use another classic construction, this time a
\textbf{$B_{2}$ }type of Ruzsa set.
\begin{example}
Let $g=2$, then for $p=11$ there is
\[
Ruzsa(11,\theta,2)=\{3,\mbox{ }7,\mbox{ }8,\mbox{ }10,\mbox{ }31,\mbox{ }37,\mbox{ }39,\mbox{ }45,\mbox{ }46,\mbox{ }49\}\in B_{2}[2]\left(\mathbb{Z}_{55}\right),
\]
because of the main result applied to the set
\[
Ruzsa(11,\theta)=\{7,\mbox{ }39,\mbox{ }58,\mbox{ }63,\mbox{ }65,\mbox{ }86,\mbox{ }92,\mbox{ }100,\mbox{ }101,\mbox{ }104\}\in B_{2}(\mathbb{Z}_{110})
\]
where $\theta=2$ is a primitive element of $\mathbb{F}_{11}.$ In this
case $40\in\mathbb{Z}_{55}$ has two \linebreak different representations as
the sum of two elements from the set $Ruzsa(11,\theta,2)$, $$3+37\equiv46+49(\bmod\,55).$$

Let $g=5$, then for $p=11$ and $\theta=2$
\[
Ruzsa(11,\theta,5)=\{4,\mbox{ }7,\mbox{ }12,\mbox{ }13,\mbox{ }14,\mbox{ }16,\mbox{ }17,\mbox{ }19,\mbox{ }20,\mbox{ }21\}\in B_{2}[5]\left(\mathbb{Z}_{22}\right),
\]
Here we could check directly that
\[
13+20\equiv4+7\equiv19+14\equiv16+17\equiv21+12(\bmod\,22).
\]

\end{example}
Now we are going to use a $B_{5}$ type of Gómez and Trujillo set
to obtained $B_{5}[g]$ sets for at least two different values of
$g$.
\begin{example}
Let $g=2$, then for $p=5$ and $\theta=x^3+x-1$ primitive element of $\mathbb{F}_{5^4}$, there is
\[
G\&T(5,5,\theta,2)=\{127,\mbox{ }226,\mbox{ }258,\mbox{ }625,\mbox{ }1384\}\in B_{5}[2]\left(\mathbb{Z}_{1560}\right),
\]
because of the main result applied to the set
\[
G\&T(5,5,\theta)=\{226,\mbox{ }625,\mbox{ }1384,\mbox{ }1687,\mbox{ }1818\}\in B_{5}(\mathbb{Z}_{3120})
\]

Let $g=8$, then for $p=5$ applying the main result to the previous
$B_{5}$, obtained
\[
G\&T(5,5,\theta,8)=\{127,\mbox{ }214,\mbox{ }226,\mbox{ }235,\mbox{ }258\}\in B_{5}[8]\left(\mathbb{Z}_{390}\right),
\]
that is in fact a $B_{5}[2]$ by direct verification.
\end{example}
The main advantage over Martin and O'Bryant way to constructed the
Generalized Sidon sets is that we could construct $B_{h}[g]$ sets
for any integer $h\geq2$ and any integer $g\geq2$, just using the
correct construction and the appropriate prime $p$ or prime power
$q$ as we showed on the previous examples.

\section{Open problems}

\begin{enumerate}
\item On this paper we used $B_{h}$ sets coming from some classic constructions
to built $B_{h}[g]$ sets, it is clear we do not only present a result
that allows to do this but also we study the appropriate homomorphisms
for those constructions. Another classic construction that we could
take into consideration is the one created for Singer of a $B_{2}$
set mainly its generalization to built $B_{h}$ sets \cite{B&C}.
\\
Is still an open problem to study the suitable homomorphism to use
with the $B_{h}$ set type Singer. This problem required the study
of the factorization of the $h^{th}$ cyclotomic polynomial $\frac{q^{h+1}-1}{q-1}$
with $q$ a prime power.
\item The examples on this paper showed some cases where the homomorphism
$\varphi:G\rightarrow G^{\prime}$ send a $B_{h}$ set into a $B_{h}[g^{\prime}]$
with $g^{\prime}<|ker(\varphi)|$, this generally happens for large
$|Ker(\varphi)|$. The problem is to identify the structure of the
$B_{h}$ set in order to obtain a $B_{h}[g]$ set using our main result
with the maximum and minimum possible value of $g$ that we can get.
\item The study of functions not homomorphisms that preserve the property
of been a $B_{h}$ set, the interest on this problem is to identify
the functions that transform a $B_{h}$ set into a $B_{h}[g]$ set,
at least when the ambient group is $\mathbb{Z}$.
\item For the constructions of $B_{h}$ sets that we did not consider at
this paper, to study the appropriate homomorphism that preserve the
cardinality of the set.\end{enumerate}

\textbf{Acknowledgement.} We thank the Universidad del Valle for the financial support provided to complete the postgraduate studies from the authors Y. Caicedo and J. Gómez. Also, the authors thank to  ``Patrimonio Autónomo Fondo Nacional de Financiamiento para la Ciencia, la Tecnología y la Innovación - Francisco José de Caladas'', Universidad del Cauca and COLCIENCIAS for financial support under research project 110356935047 titled ``Construcciones de Conjuntos $B_h[g]$, propiedades de Midy, y algunas aplicaciones''.

\end{document}